\documentclass[oneside,10pt]{article}
\usepackage[b5paper]{geometry}	

\usepackage{amsfonts,amsmath,latexsym,amssymb}
\usepackage{mathrsfs,upref}
\usepackage{mathptmx}		    	                            		
\usepackage{amscd}
\usepackage{amsthm}
\usepackage{graphicx}
\usepackage{amsmath}
\usepackage{lipsum}

\newcommand\blfootnote[1]{%
  \begingroup
  \renewcommand\thefootnote{}\footnote{#1}%
  \addtocounter{footnote}{-1}%
  \endgroup}

\newtheorem{theorem}{Theorem}

\newtheorem{definition}[theorem]{Definition}

\newtheorem{lemma}[theorem]{Lemma}

\newtheorem{proposition}[theorem]{Proposition}
\newtheorem{remark}[theorem]{Remark}

\DeclareMathOperator\supp{supp}
\newcommand\esssup{\mathop{\rm ess \, sup}}
\newcommand\essinf{\mathop{\rm ess \, inf}}

\begin{document}

\title{Variable Calder\'on-Hardy spaces on the Heisenberg group}
\author{Pablo Rocha}

\maketitle

\begin{abstract}
Let $\mathbb{H}^{n}$ be the Heisenberg group and $Q = 2n+2$. For $1 < q < \infty$, $\gamma > 0$ and an exponent function $p(\cdot)$ on 
$\mathbb{H}^n$, which satisfy log-H\"older conditions, with $0 < p_{-} \leq p_{+} < \infty$, we introduce the variable 
Calder\'on-Hardy spaces $\mathcal{H}^{p(\cdot)}_{q, \gamma}(\mathbb{H}^{n})$, and show for every $f \in H^{p(\cdot)}(\mathbb{H}^{n})$ that the equation
\[
\mathcal{L} F = f
\]
has a unique solution $F$ in $\mathcal{H}^{p(\cdot)}_{q, 2}(\mathbb{H}^{n})$, where $\mathcal{L}$ is the sublaplacian on $\mathbb{H}^{n}$, 
$1 < q < \frac{n+1}{n}$ and $Q (2 + \frac{Q}{q})^{-1} < \underline{p}$.
\end{abstract}


\section{Introduction}

The  Heisenberg group $\mathbb{H}^{n}$ can be identified with $\mathbb{R}^{2n} \times \mathbb{R}$ whose group law 
(noncommutative) is given by
\begin{equation} \label{group law}
(x,t) \cdot (y,s) = \left( x+y, t+s + x^{t} J y \right),
\end{equation}
where $J$ is the $2n \times 2n$ skew-symmetric matrix given by
\[
J= 2 \left( \begin{array}{cc}
                           0 & -I_n \\
                           I_n & 0 \\
                                      \end{array} \right)
\]
being $I_n$ the $n \times n$ identity matrix.

The dilation group on $\mathbb{H}^n$ is defined by 
\[
r \cdot (x,t) = (rx, r^{2}t), \,\,\,\ r > 0.
\]
With this structure we have that $e = (0,0)$ is the neutral element, $(x, t)^{-1}=(-x, -t)$ is the inverse of $(x, t)$, and
$r \cdot((x,t) \cdot (y,s)) = (r\cdot(x,y)) \cdot (r\cdot(y,s))$. 

The {\it Koranyi norm} on $\mathbb{H}^{n}$ is the function $\rho : \mathbb{H}^{n} \to [0, \infty)$ defined by
\begin{equation} \label{Koranyi norm}
\rho(x,t) = \left( |x|^{4} + \, t^{2}  \right)^{1/4}, \,\,\, (x,t) \in \mathbb{H}^{n},
\end{equation}
where $| \cdot |$ is the usual Euclidean norm on $\mathbb{R}^{2n}$. Moreover, $\rho$ is continuous on $\mathbb{H}^{n}$ and is smooth on 
$\mathbb{H}^{n} \setminus \{ e \}$. 

The $\rho$ - ball centered at $z_0 \in \mathbb{H}^{n}$ with radius $\delta > 0$ is defined by
\[
B(z_0, \delta) := \{ w \in \mathbb{H}^{n} : \rho(z_0^{-1} \cdot w) < \delta \}.
\]

The topology in $\mathbb{H}^{n}$ induced by the $\rho$ - balls coincides with the Euclidean topology of 
$\mathbb{R}^{2n} \times \mathbb{R} \equiv\mathbb{R}^{2n+1}$ (see \cite[Proposition 3.1.37]{Fischer}). So, the borelian sets of 
$\mathbb{H}^{n}$ are identified with those of $\mathbb{R}^{2n+1}$. The Haar measure in $\mathbb{H}^{n}$ is the Lebesgue measure of 
$\mathbb{R}^{2n+1}$, thus $L^{p}(\mathbb{H}^{n}) \equiv L^{p}(\mathbb{R}^{2n+1})$, for every $0 < p \leq \infty$. Moreover, for 
$f \in L^{1}(\mathbb{H}^{n})$ and for $r > 0$ fixed, we have
\begin{equation} \label{homog dim}
\int_{\mathbb{H}^{n}} f(r \cdot z) \, dz = r^{-Q} \int_{\mathbb{H}^{n}} f(z) \, dz,
\end{equation}
where $Q= 2n+2$. The number $2n+2$ is known as the {\it homogeneous dimension} of $\mathbb{H}^{n}$ (we observe that the {\it topological dimension} of $\mathbb{H}^{n}$ is $2n+1$).

If $f$ and $g$ are measurable functions on $\mathbb{H}^{n}$, their convolution $f \ast g$ is defined by
\[
(f \ast g)(z) := \int_{\mathbb{H}^{n}} f(w) g(w^{-1} \cdot z) \, dw,
\]
when the integral is finite.

A measurable function $p(\cdot) : \mathbb{H}^{n} \rightarrow (0, \infty)$ is called an exponent on $\mathbb{H}^n$, we adopt the standard notation in variable exponents. Given a measurable set $E \subset \mathbb{H}^{n}$, let
\[
p_{-}(E) = \essinf_{ z \in E } p(z), \,\,\,\, \text{and} \,\,\,\, p_{+}(E) = \esssup_{z \in E} p(z).
\]
When $E = \mathbb{H}^{n}$, we will simply write $p_{-} := p_{-}(\mathbb{H}^{n})$ and $p_{+} := p_{+}(\mathbb{H})$.
Throughout this paper, we will assume that $0 < p_{-} \leq p_{+} < \infty$. We also define $\underline{p}=\min \left\{ p_{-},1\right\}$.

Given an exponent $p(\cdot): \mathbb{H}^{n} \rightarrow (0, \infty)$, on the set of the all measurable function $f$, we define the modular function $\kappa_{p(\cdot)}$ by
\begin{equation} \label{mod}
\kappa_{p(\cdot)}(f) = \int_{\mathbb{H}^{n}} \, |f(z)|^{p(z)} \, dz.
\end{equation}
By $L^{p(\cdot)}(\mathbb{H}^{n})$ we denote the space of all measurable functions $f$ 
on $\mathbb{H}^n$ such that for some $\lambda > 0$,
\[
\kappa_{p(\cdot)}(f/\lambda) < \infty.
\] 
We set 
\[
\| f \|_{p(\cdot)} = \inf \left\{ \lambda > 0 : \kappa_{p(\cdot)}(f/\lambda) \leq 1 \right\}.
\]
We see that $\left( L^{p(\cdot)}(\mathbb{H}^{n}), \| . \|_{p(\cdot)} \right)$ is a quasi normed space.
These spaces are referred to as the Lebesgue spaces with variable exponents or variable Lebesgue spaces (see \cite{diening}).

In \cite{Foll-St}, for $0 < p < \infty$, G. Folland and E. M. Stein defined the Hardy Spaces $H^{p}(\mathbb{H}^{n})$ on the Heisenberg group with the norm given by
\[
\| f \|_{H^{p}(\mathbb{H}^{n})} = \left\| \sup_{t>0} \sup_{\phi \in \mathcal{F}_{N}} \left| f \ast \phi_t \right| \right\|_{p},
\]
where $\phi_t(z) = t^{-Q} \phi(t^{-1} \cdot z)$ with $t > 0$ and $\mathcal{F}_{N}$ is a suitable family of smooth functions. In the paper \cite{Fang}, J. Fang and J. Zhao defined the variable Hardy spaces on the Heisenberg group $H^{p(\cdot)}(\mathbb{H}^{n})$, replacing $L^{p}$ by $L^{p(\cdot)}$ in the above norm and they investigate their several properties.

Let $L^{q}_{loc}(\mathbb{H}^{n})$, $1 < q < \infty$, be the space of all measurable functions $g$ on $\mathbb{H}^{n}$ that belong locally to $L^{q}$ for compact sets of $\mathbb{H}^{n}$. We endowed $L^{q}_{loc}(\mathbb{H}^{n})$ with the topology generated by the seminorms
$$|g|_{q, \, B} = \left( |B|^{-1} \int_{B} \, |g(w)|^{q}\, dw \right)^{1/q},$$ where $B$ is a $\rho$-ball in $\mathbb{H}^{n}$ and $|B|$ denotes its Haar measure.

For $g \in L^{q}_{loc}(\mathbb{H}^{n})$, we define a maximal function $\eta_{q, \, \gamma}(g; z)$ as
\begin{equation} \label{etamax}
\eta_{q, \, \gamma}(g; \, z) = \sup_{r > 0} r^{-\gamma} |g|_{q, \, B(z, r)},
\end{equation}
where $\gamma$ is a positive real number and $B(z, r)$ is the $\rho$-ball centered at $z$ with radius $r$.

Let $k$ a non negative integer and $\mathcal{P}_{k}$ the subspace of $L^{q}_{loc}(\mathbb{H}^{n})$ formed by all the polynomials of homogeneous degree at most $k$. We denote by $E^{q}_{k}$ the quotient space of $L^{q}_{loc}(\mathbb{H}^{n})$ by $\mathcal{P}_{k}$. If 
$G \in E^{q}_{k}$, we define the seminorm $\| G \|_{q, \, B} = \inf \left\{ |g|_{q, \, B} : g \in G \right\}$. The family of all these seminorms induces on $E^{q}_{k}$ the quotient topology.

Given a positive real number $\gamma$, we can write $\gamma = k + t$, where $k$ is a non negative integer and $0 < t \leq 1$. This decomposition is unique.

For $G \in E^{q}_{k}$, we define a maximal function $N_{q, \, \gamma}(G; z)$ as 
\begin{equation} \label{Nmax}
N_{q, \, \gamma}(G; z) = \inf \left\{ \eta_{q, \, \gamma}(g; z) : g \in G \right\},
\end{equation}
such maximal function is lower semicontinuous, see \cite[Lemma 9]{rocha5}.

\begin{definition} \label{var cal-hardy}
Let $p(\cdot) : \mathbb{H}^{n} \rightarrow (0, \infty)$ be an exponent such that $0< p_{-} \leq p_{+} < \infty$, we say 
that an element $G \in E^{q}_{k}$ belongs to the variable Calder\'on-Hardy space on the Heisenberg group 
$\mathcal{H}^{p(\cdot)}_{q, \, \gamma}(\mathbb{H}^{n})$ if the maximal function 
$N_{q, \, \gamma}(G; \cdot) \in L^{p(\cdot)}(\mathbb{H}^{n})$. The "norm" of $G$ in $\mathcal{H}^{p(\cdot)}_{q, \, \gamma}(\mathbb{H}^{n})$ is defined as $\| G \|_{\mathcal{H}^{p(\cdot)}_{q, \, \gamma}(\mathbb{H}^{n})} = \| N_{q, \, \gamma}(G; \cdot) \|_{p(\cdot)}$.
\end{definition}

The Calder\'on-Hardy spaces were defined in the setting of the classical Lebesgue spaces by A. B. Gatto, J. G. Jim\'enez and C. Segovia in \cite{segovia}, they characterize the solution of the equation $\Delta^{m} F = f$, $m \in \mathbb{N}$, for $f \in H^{p}(\mathbb{R}^{n})$. Moreover, they proved that the iterated Laplace operator $\Delta^{m}$ is a bijective mapping from the Calder\'on-Hardy spaces onto $H^{p}(\mathbb{R}^{n})$. 

The equation $\Delta^{m} F = f$, $m \in \mathbb{N}$, for $f \in H^{p(\cdot)}(\mathbb{R}^{n})$ and $f \in H^{p}(\mathbb{R}^{n}, w)$, 
was studied by the author in \cite{rocha1} and \cite{rocha3} respectively, obtaining analogous results to those of Gatto, Jim\'enez and Segovia. %(see also \cite{rocha4}).

Lately, Z. Liu, Z. He and H. Mo in \cite{He} extended the definition of Calder\'on-Hardy spaces to Orlicz setting. These new Orlicz 
Calder\'on-Hardy spaces can cover classical Calder\'on-Hardy spaces in \cite{segovia}. As an application, they solved the equation 
$\Delta^{m} F = f$ when $f \in H^{\Phi}(\mathbb{R}^{n})$, where $H^{\Phi}(\mathbb{R}^{n})$ are the Orlicz-Hardy spaces defined 
in \cite{Nakai}.

Recently, the author in \cite{rocha5} studied an analogous problem on the Heisenberg group. More precisely, we proved that the equation
$\mathcal{L} F = f$ for $f \in H^p(\mathbb{H}^n)$ has a unique solution $F$ in $\mathcal{H}^{p}_{q, 2}(\mathbb{H}^{n})$, where 
$\mathcal{L}$ is the sublaplacian on $\mathbb{H}^{n}$, $1 < q < \frac{n+1}{n}$ and $Q (2 + \frac{Q}{q})^{-1} < p \leq 1$.

The purpose of this work is to extend the results obtained by the author in \cite{rocha5} to the variable setting. For them, we must take into account certain aspects inherent to the variable spaces.

We say that an exponent function $p(\cdot) : \mathbb{H}^{n} \to (0, \infty)$ such that $0 < p_{-} \leq p_{+} < \infty$ belongs 
to $\mathcal{P}^{\log}(\mathbb{H}^{n})$, if there exist positive constants $C$, $C_{\infty}$ and $p_{\infty}$ such that $p(\cdot)$ satisfies the local log-H\"older continuity condition, i.e.:
\[
|p(z) - p(w)| \leq \frac{C}{-\log(\rho(z^{-1} \cdot w))}, \,\,\, \text{for} \,\, \rho(z^{-1} \cdot w) \leq \frac{1}{2},
\]
and is log-H\"older continuous at infinity, i.e.:
\[
|p(z) - p_{\infty}| \leq \frac{C_{\infty}}{\log(e+\rho(z))}, \,\,\, \text{for all} \,\, z \in \mathbb{H}^{n}.
\]
Here $\rho$ is the {\it Koranyi norm} given by (\ref{Koranyi norm}).

Let $G \in \mathcal{H}^{p(\cdot)}_{q, \, 2}(\mathbb{H}^{n})$, then $N_{q, \, 2}(G; z_0) < \infty$ $\text{a.e.} \,\, z_0 \in \mathbb{H}^{n}$. 
By Lemma \ref{puntual 1}-(i) below, there exists $g \in G$ such that $N_{q, \, 2}(G; z_0) = \eta_{q, \, 2}(g; z_0)$. Now, from 
Proposition \ref{g distrib} below it follows that $g  \in \mathcal{S}'(\mathbb{H}^{n})$. So the sublaplacian of $g$, $\mathcal{L}g$, is well defined in sense of distributions. On the other hand, since any two representatives of $G$ differ in a polynomial of homogeneous degree at most $1$, we have that $\mathcal{L}g$ is independent of the representative $g \in G$ chosen. Therefore, for $G \in E^{q}_{1}$, we shall define 
$\mathcal{L}G$ as the distribution $\mathcal{L}g$, where $g$ is any representative of $G$.

Our main results are the following.

\

{\bf Theorem \ref{main thm}.} \textit{Let $p(\cdot)$ be an exponent that belongs to $\mathcal{P}^{\log}(\mathbb{H}^{n})$ and 
$1 < q < \frac{n+1}{n}$. If $\underline{p} > Q (2 + Q/q)^{-1}$, then the sublaplacian $\mathcal{L}$ is a bijective mapping from 
$\mathcal{H}^{p(\cdot)}_{q, 2}(\mathbb{H}^{n})$ onto $H^{p(\cdot)}(\mathbb{H}^{n})$. Moreover, there exist two positive constants 
$c_1$ and $c_2$ such that
\[
c_1 \|G \|_{\mathcal{H}^{p(\cdot)}_{q, 2}(\mathbb{H}^{n})} \leq \| \mathcal{L}G \|_{H^{p(\cdot)}(\mathbb{H}^{n})} \leq 
c_2 \|G \|_{\mathcal{H}^{p(\cdot)}_{q, 2}(\mathbb{H}^{n})}
\]
hold for all $G \in \mathcal{H}^{p(\cdot)}_{q, 2}(\mathbb{H}^{n})$.}

\

The case $p_{+} \leq Q (2 + Q/q)^{-1}$ is trivial.

\

{\bf Theorem \ref{second thm}} \textit{If $p(\cdot)$ is an exponent function on $\mathbb{H}^n$ such that $p_{+} \leq Q (2 + Q/q)^{-1}$, then 
$\mathcal{H}^{p(\cdot)}_{q, \, 2}(\mathbb{H}^{n}) = \{ 0 \}.$}

\

In Section 2 we state the basics of the Heisenberg group and variable Lebesgue spaces together with some auxiliary lemmas and propositions related to variable Calder\'on-Hardy and variable Hardy spaces on the Heisenberg group. We also recall the definition and atomic decomposition of variable Hardy spaces on $\mathbb{H}^n$ given in \cite{Fang}. Finally, in Section 3 we prove our main results.

\

\textbf{Notation:} The symbol $A \lesssim B$ stands for the inequality $A \leq c B$ for some constant $c$, and $A \sim B$ stands for 
$B \lesssim A \lesssim B$. We denote by $B(z_0, \delta)$ the $\rho$ - ball centered at $z_0 \in \mathbb{H}^{n}$ with radius $\delta$. 
Given $\beta > 0$ and a $\rho$ - ball $B= B(z_0, \delta)$, we set $\beta B= B(z_0, \beta \delta)$. For a measurable subset 
$E\subseteq \mathbb{H}^{n}$ we denote by $\left\vert E\right\vert $ and $\chi_{E}$ the Haar measure of $E$ and the characteristic 
function of $E$ respectively. 

Throughout this paper, $C$ will denote a positive constant, not necessarily the same at each occurrence.

\section{Preliminaries}

\subsection{Basics on Heisenberg group}

Let $\mathbb{H}^n$ be the Heisenberg group with group law given by (\ref{group law}). 
If $B(z_0, \delta)$ is a $\rho$ - ball of $\mathbb{H}^n$, then its Haar measure is
\[
|B(z_0, \delta)| = c \delta^{Q},
\]
where $c = |B(e,1)|$ and  $Q = 2n+2$. Given $\lambda > 0$, we put $\lambda B = \lambda B(z_0, \delta) = 
B(z_0, \lambda \delta)$. So $|\lambda B| = \lambda^{Q}	|B|$.

The Hardy-Littlewood maximal operator $M$ on the Heisenberg group is defined by
\[
Mf(z) = \sup_{B \ni z} |B|^{-1}\int_{B} |f(w)| \, dw,
\]
where $f$ is a locally integrable function on $\mathbb{H}^{n}$ and the supremum is taken over all the $\rho$ - balls $B$ containing $z$.

For every $i = 1,2, ..., 2n+1$, $X_i$ denotes the left invariant vector field on $\mathbb{H}^n$ given by
\[
X_i = \frac{\partial}{\partial x_i} + 2 x_{i+n} \frac{\partial}{\partial t}, \,\,\,\, i=1, 2, ..., n;
\]
\[
X_{i+n} = \frac{\partial}{\partial x_{i+n}} - 2 x_{i} \frac{\partial}{\partial t}, \,\,\, i=1, 2, ..., n;
\]
and
\[
X_{2n+1} = \frac{\partial}{\partial t}.
\]
The sublaplacian on $\mathbb{H}^n$, denoted by $\mathcal{L}$, is the counterpart of the Laplacain $\Delta$ on $\mathbb{R}^n$. The sublaplacian $\mathcal{L}$ is defined by
\[
\mathcal{L} =- \sum_{i=1}^{2n} X_i^2,
\]
where $X_i$, $i= 1, ..., 2n$, are the left invariant vector fields defined above.

Given a multi-index $I=(i_1,i_2, ..., i_{2n}, i_{2n+1}) \in (\mathbb{N} \cup \{ 0 \})^{2n+1}$, we set
\[
|I| = i_1 + i_2 + \cdot \cdot \cdot + i_{2n} + i_{2n+1}, \hspace{.5cm} d(I) = i_1 + i_2 + \cdot \cdot \cdot + i_{2n} + 2 \, i_{2n+1}.
\]
The amount $|I|$ is called the length of $I$ and $d(I)$ the homogeneous degree of $I$. We adopt the following multi-index notation for
higher order derivatives and for monomials on $\mathbb{H}^{n}$. If $I=(i_1, i_2, ..., i_{2n+1})$ is a multi-index, 
$X = \{ X_i \}_{i=1}^{2n+1}$, and $z = (x,t) = (x_1, ..., x_{2n}, t) \in \mathbb{H}^{n}$, we put
\[
X^{I} := X_{1}^{i_1} X_{2}^{i_2} \cdot \cdot \cdot X_{2n+1}^{i_{2n+1}}, \,\,\,\,\,\, \text{and} \,\,\,\,\,\,
z^{I} := x_{1}^{i_1} \cdot \cdot \cdot x_{2n}^{i_{2n}} \cdot t^{i_{2n+1}}.
\]
Every polynomial $p$ on $\mathbb{H}^n$ can be written as a unique finite linear combination of the monomials $z^I$, that is
\begin{equation} \label{polynomial}
p(z) = \sum_{I \in \mathbb{N}_{0}^n} c_I z^I,
\end{equation}
where all but finitely many of the coefficients $c_I \in \mathbb{C}$ vanish. The \textit{homogeneous degree} of a polynomial $p$ written as 
(\ref{polynomial}) is $\max \{ d(I) : I \in \mathbb{N}_{0}^n \,\, \text{with} \,\, c_I \neq 0 \}$. Let $k \in \mathbb{N} \cup \{ 0 \}$, we 
remember that $\mathcal{P}_{k}$ denotes the subspace formed by all the polynomials of homogeneous degree at most $k$. So, every 
$p \in \mathcal{P}_{k}$ can be written as $p(z) = \sum_{d(I) \leq k} c_I \, z^I$, with $c_I \in \mathbb{C}$.

The Schwartz space $\mathcal{S}(\mathbb{H}^{n})$ is defined by
\[
\mathcal{S}(\mathbb{H}^{n}) = \left\{ \phi \in C^{\infty}(\mathbb{H}^{n}) : \sup_{z \in \mathbb{H}^{n}} (1+\rho(z))^{N} 
|(X^{I} \phi)(z)| < \infty \,\,\, \forall \,\, N \in \mathbb{N}_{0}, \, I \in (\mathbb{N}_{0})^{2n+1}  \right\}.
\]
We topologize the space $\mathcal{S}(\mathbb{H}^{n})$ with the following family of seminorms
\[
\| \phi \|_{\mathcal{S}(\mathbb{H}^{n}), \, N} = \sum_{d(I) \leq N} \sup_{z \in \mathbb{H}^{n}} (1+\rho(z))^{N} |(X^{I} \phi)(z)| \,\,\,\,\,\,\, (N \in \mathbb{N}_{0}),
\]
with $\mathcal{S}'(\mathbb{H}^{n})$ we denote the dual space of $\mathcal{S}(\mathbb{H}^{n})$. 

\subsection{Basics on variable Lebesgue spaces}

Given an exponent $p(\cdot): \mathbb{H}^{n} \rightarrow (0, \infty)$, we consider the variable Lebesgue space $L^{p(\cdot)}(\mathbb{H}^{n})$ defined above. It not so hard to see the following,

\

$1.$ $\| f \|_{p(\cdot)} \geq 0$, and $\| f \|_{p(\cdot)}=0$ if and only if $f \equiv 0$.

$2.$ $\| c \, f \|_{p(\cdot)} = |c| \, \| f \|_{p(\cdot)}$ for $c \in \mathbb{C}$.

$3.$ $\| f + g \|_{p(\cdot)}^{\underline{p}} \leq \| f \|_{p(\cdot)}^{\underline{p}} + \| g \|_{p(\cdot)}^{\underline{p}}$. 

$4.$ $\| f \|_{p(\cdot)}^s = \| |f|^s \|_{p(\cdot)/s}$, for every $s > 0$. \\
$\,\,$ \\
A direct consequence of $\underline{p}$-triangle inequality is the quasi-triangle inequality
$$\| f + g \|_{p(\cdot)} \leq 2^{1/\underline{p} - 1}\left(\| f \|_{p(\cdot)} + \| g \|_{p(\cdot)}\right),$$
for all $f, \, g \in L^{p(\cdot)}(\mathbb{H}^{n})$.

\

The following Fefferman-Stein vector-valued maximal inequality on $L^{p(\cdot)}(\mathbb{H}^{n})$ was proved in \cite{Fang}.

\begin{theorem} (\cite[Theorem 4.2]{Fang}) \label{vector max ineq}
Let $p(\cdot) \in \mathcal{P}^{\log}(\mathbb{H}^n)$ with $1 < p_{-} \leq p_{+} < \infty$. Then for every $\theta \in
\left( 1,\infty \right)$, we have
\[
\left\Vert \left( \sum\limits_{j=1}^{\infty }\left( M f_{j}\right)
^{\theta }\right) ^{\frac{1}{\theta }}\right\Vert_{L^{p(\cdot)}(\mathbb{H}^{n})} \lesssim \left\Vert
\left( \sum\limits_{j=1}^{\infty }\left\vert f_{j}\right\vert ^{\theta
}\right) ^{\frac{1}{\theta }}\right\Vert_{L^{p(\cdot)}(\mathbb{H}^{n})},
\]%
for all sequences of bounded measurable functions with compact support $\left\{ f_{j}\right\}
_{j=1}^{\infty}$.
\end{theorem}

The following two results refer to the modular function $\kappa_{p(\cdot)}$ given by (\ref{mod}).

\begin{lemma} \label{modular} Let $p(\cdot)$ be an exponent on $\mathbb{H}^n$ such that $0 < p_{-} \leq p_{+} < \infty$. Then 
$f \in L^{p(\cdot)}(\mathbb{H}^{n})$ if and only if $\kappa_{p(\cdot)}(f) < \infty$.
\end{lemma}

\begin{proof} Clearly, if $\kappa_{p(\cdot)}(f) < \infty$, then $f\in L^{p(\cdot)}(\mathbb{H}^{n})$. Conversely, if 
$f\in L^{p(\cdot)}(\mathbb{H}^{n})$, then we have that $\kappa_{p(\cdot)}(f / \lambda) < \infty$ for some $\lambda >1$. Then
\[
\kappa_{p(\cdot)}(f)= \int_{\mathbb{H}^{n}} \, \left| \frac{\lambda f(z)}{\lambda} \right|^{p(z)} \, dz \leq 
\lambda^{p_{+}} \kappa_{p(\cdot)}(f / \lambda) < \infty.
\]
\end{proof}

\begin{lemma} \label{modular a cero} Let $p(\cdot)$ be an exponent on $\mathbb{H}^n$ such that $0 < p_{-} \leq p_{+} < \infty$. If $\{ f_j \}$ is a sequence of measurable functions on $\mathbb{H}^n$ such that $\kappa_{p(\cdot)}(f_j) \rightarrow 0$, then $\| f_j \|_{p(\cdot)} \rightarrow 0$.
\end{lemma}

\begin{proof} Suppose that $\kappa_{p(\cdot)}(f_j) \rightarrow 0$. Given $0 < \epsilon < 1$ for sufficiently large $j$ we have 
$\kappa_{p(\cdot)}(f_j) \leq \epsilon$ and so
\[
\kappa_{p(\cdot)}\left(f_j \kappa_{p(\cdot)}(f_j)^{-1/p_{+}}\right) \leq \kappa_{p(\cdot)}(f_j)^{-1} \kappa_{p(\cdot)}(f_j) =1,
\]
from this it follows that $\| f_j \|_{p(\cdot)} \leq \kappa_{p(\kappa)}(f_j)^{1/p_{+}} \leq \epsilon^{1/p_{+}}$. Thus, 
$\| f_j \|_{p(\cdot)} \rightarrow 0$.
\end{proof}

\subsection{Variable Calder\'on-Hardy spaces on $\mathbb{H}^{n}$}

Let $1 < q < \infty$ and $\gamma > 0$. In this section we establish some results concerning to the maximal functions 
$\eta_{q, \, \gamma} (g; \cdot)$ and $N_{q, \, \gamma}(G; \cdot)$ defined in (\ref{etamax}) and (\ref{Nmax}) respectively. We recall that 
the maximal function $N_{q, \, \gamma}(G; \cdot)$ is used to define the variable Calder\'on-Hardy spaces on $\mathbb{H}^{n}$ 
(see Definition \ref{var cal-hardy}).

\begin{lemma} \label{puntual 1} Let $G \in E^{q}_{k}$ with $N_{q, \, \gamma}(G; z_0) < \infty,$ for some $z_0 \in \mathbb{H}^{n}$. Then:

$(i)$ There exists a unique $g \in G$ such that $\eta_{q, \, \gamma} (g; z_0) < \infty$ and, therefore, 
$\eta_{q, \, \gamma} (g; z_0) = N_{q, \, \gamma}(G; z_0)$.

$(ii)$ For any $\rho$-ball $B$, there is a constant $c$ depending on $z_0$ and $B$ such that if $g$ is the unique representative of $G$ given in $(i)$, then
$$\|G\|_{q, \, B} \leq |g|_{q, \, B} \leq c \, \eta_{q, \, \gamma} (g; z_0) = c \, N_{q, \, \gamma}(G; z_0).$$

The constant $c$ can be chosen independently of $z_0$ provided that $z_0$ varies in a compact set.
\end{lemma}
\begin{proof} The proof is similar to the one given in \cite[Lemma 3]{segovia}.
\end{proof}

\begin{lemma} \label{series in Eqk} Let $\{ G_{j} \}$ be a sequence in $E^{q}_{k}$ such that for a given point $z_0 \in \mathbb{H}^n$, the series 
$\sum_j N_{q, \, \gamma}(G_{j}; \, z_0 )$ is finite. Then

$(i)$ The series $\sum_j G_j$ converges in $E_{k}^{q}$ to an element $G$ and 
$$N_{q, \, \gamma}(G; \, z_0 ) \leq \sum_j N_{q, \, \gamma}(G_{j}; \, z_0 ).$$

$(ii)$ If $g_j$ is the unique representative of $G_j$ satisfying
$\eta_{q, \, \gamma} (g_j; z_0) = N_{q, \, \gamma}(G_j; z_0)$, then $\sum_j g_j$ converges in $L^{q}_{loc}(\mathbb{H}^{n})$ to a function 
$g$ that is the unique representative of $G$ satisfying $\eta_{q, \, \gamma} (g; z_0) = N_{q, \, \gamma}(G; z_0)$
\end{lemma}

\begin{proof} The proof is similar to the one given in \cite[Lemma 4]{segovia}.
\end{proof}

\begin{proposition} (\cite[Proposition 15]{rocha5}) \label{g distrib}
If $g \in L^{q}_{loc}(\mathbb{H}^{n})$, $1 < q < \infty$, and there is a point $z_0 \in \mathbb{H}^{n}$ such that 
$\eta_{q, \, \gamma} (g ; z_0) < \infty$, then $g \in \mathcal{S}'(\mathbb{H}^{n})$.
\end{proposition}

Given an exponent $p(\cdot): \mathbb{H}^{n} \rightarrow (0, \infty)$, $1 < q < \infty$ and $\gamma > 0$, we consider the variable 
Calder\'on-Hardy  space $\mathcal{H}^{p(\cdot)}_{q, \, \gamma}(\mathbb{H}^{n})$. The following result states the completeness of variable Calder\'on-Hardy spaces.

\begin{proposition} \label{completeness} The space $\mathcal{H}^{p(\cdot)}_{q, \, \gamma}(\mathbb{H}^{n})$, $0 < p_{-} \leq p_{+} < \infty$, is complete.
\end{proposition}

\begin{proof} 
The proof is similar to the one given in \cite[Proposition 9]{rocha1}.
\end{proof}

\subsection{Variable Hardy spaces on $\mathbb{H}^{n}$}

In the paper \cite{Fang}, J. Fang and J. Zhao give a variety of distinct approaches, based on differing definitions, all lead to the same notion of variable Hardy space $H^{p(\cdot)}(\mathbb{H}^n)$.

We recall some terminologies and notations from the study of maximal functions used in \cite{Fang}. Given $N \in \mathbb{N}$, define 
\[
\mathcal{F}_{N}=\left\{ \phi \in \mathcal{S}(\mathbb{H}^{n}):\sum\limits_{d(I) \leq N}
\sup\limits_{z \in \mathbb{H}^{n}}\left( 1 + \rho(z) \right)^{N} | (X^{I} \phi) (z) | \leq 1\right\}.
\]
For any $f \in \mathcal{S}'(\mathbb{H}^{n})$, the grand maximal function of $f$ is given by 
\[
\mathcal{M}_{\mathcal{F}_{N}}f(z)=\sup\limits_{t>0}\sup\limits_{\phi\in\mathcal{F}_{N}} 
\left\vert \left( f \ast t^{-Q}\phi(t^{-1} \cdot) \right)(z) \right\vert,
\]
where $N$ is a large and fix integer.

\begin{definition} \label{Dp def} Given an exponent function $p(\cdot):\mathbb{H}^{n} \to ( 0,\infty)$ with $0 < p_{-} \leq p_{+} < \infty$, we define the integer $\mathcal{D}_{p(\cdot)}$ by
\[
\mathcal{D}_{p(\cdot)} := \min \{ k \in \mathbb{N} \cup \{0\} : (2n+k+3) p_{-} > 2n+2 \}.
\]
For $N \geq \mathcal{D}_{p(\cdot)} + 1$, define 
the variable Hardy space $H^{p(\cdot)}(\mathbb{H}^{n})$ to be the collection of $f \in \mathcal{S}'(\mathbb{H}^{n})$ such that
$\| \mathcal{M}_{\mathcal{F}_{N}} f \|_{L^{p(\cdot)}(\mathbb{H}^{n})} < \infty$. Then, the "norm" on the space $H^{p(\cdot)}(\mathbb{H}^{n})$ is taken to be
$\| f \|_{H^{p(\cdot)}} := \| \mathcal{M}_{\mathcal{F}_{N}} f \|_{L^{p(\cdot)}}$.
\end{definition}

\begin{remark}
For $N \geq \mathcal{D}_{p(\cdot)} + 1$, the $H^{p(\cdot)}$-norm defined above does not depend on $N$ (see e.g. \cite{Foll-St}, therein it considers the case when $p(\cdot)$ is a constant, the variable case is similar).
\end{remark}

\begin{definition} \label{atom} Let $p(\cdot):\mathbb{H}^{n} \to ( 0,\infty)$, $0 < p_{-} \leq p_{+} < \infty $, and $p_{0} > 1$. 
Fix an integer $D \geq \mathcal{D}_{p(\cdot)}$. A measurable function $a(\cdot)$ on $\mathbb{H}^{n}$ is called a $(p(\cdot), p_{0}, D)$ - atom centered at a $\rho$ - ball $B=B(z_0, \delta)$ if it satisfies the following conditions:

$a_{1})$ $\supp(a) \subset B$,

$a_{2})$ $\| a \|_{L^{p_{0}}(\mathbb{H}^{n})} \leq 
\displaystyle{\frac{| B |^{\frac{1}{p_{0}}}}{\| \chi _{B} \|_{L^{p(\cdot)}(\mathbb{H}^{n})}}}$,

$a_{3})$ $\displaystyle{\int_{\mathbb{H}^{n}}} a(z) \, z^{I} \, dz = 0$ for all multiindex $I$ such that $d(I) \leq D$.
\end{definition}

Indeed, every $(p(\cdot), p_{0}, D)$ - atom $a(\cdot)$ belongs to $H^{p(\cdot)}(\mathbb{H}^{n})$. Moreover, there exists an universal constant
$C > 0$ such that $\| a \|_{H^{p(\cdot)}} \leq C$ for all $(p(\cdot), p_{0}, D)$ - atom $a(\cdot)$.

\

The following results talk about the size of the $\rho$ - balls in the $L^{p(\cdot)}(\mathbb{H}^{n})$-norm.

\begin{lemma} \label{Lp size}
(\cite[Lemma 4.1.]{Fang}) Suppose that $p(\cdot)$ is an exponent function such that $p(\cdot) \in \mathcal{P}^{\log}(\mathbb{H}^n)$ and 
$0 < p_{-} \leq p_{+} < \infty$.
\newline
$1)$ For all $\rho$-balls $B=B(z, \delta)$ with $z\in \mathbb{H}^{n}$ and $\vert B \vert \leq 1$, we
have
\[
\left\vert B\right\vert ^{\frac{1}{p_{-}(B)}}\sim \left\vert  B \right\vert^{
\frac{1}{p_{+}(B)}}\sim \left\vert B\right\vert ^{\frac{1}{p(z)}}\sim
\left\Vert \chi _{B}\right\Vert _{L^{p(\cdot)}(\mathbb{H}^{n})}.
\]%
$2)$ For all $\rho$-balls $B=B(z, \delta)$ with $z\in \mathbb{H}^{n}$ and $\vert B \vert \geq 1$ we
have
\[
\left\vert B \right\vert ^{\frac{1}{p_{\infty }}}\sim \left\Vert \chi
_{B}\right\Vert _{L^{p(\cdot)}(\mathbb{H}^{n})}.
\]
Here the implicit constants in $\sim $ do not depend on $z$ and $r>0.$
\end{lemma}

\begin{definition} Let $p(\cdot):\mathbb{H}^{n} \to ( 0,\infty)$ be an exponent such that $0 < p_{-} \leq p_{+} < \infty$.
Given a sequence of nonnegative numbers $\{ \lambda_j \}_{j=1}^{\infty}$ and a family of $\rho$ - balls $\{ B_j \}_{j=1}^{\infty}$, we define
\begin{equation} \label{cantidad A}
\mathcal{A} \left( \{ \lambda_j \}_{j=1}^{\infty}, \{ B_j \}_{j=1}^{\infty}, p(\cdot) \right) := 
\left\| \left\{ \sum_{j=1}^{\infty} \left( \frac{\lambda_j  \chi_{B_j}}{\| \chi_{B_j} \|_{L^{p(\cdot)}}} \right)^{\underline{p}} 
\right\}^{1/\underline{p}} \right\|_{L^{p(\cdot)}}.
\end{equation}
The space $H_{atom}^{p(\cdot),p_{0},D}\left( \mathbb{H}^{n}\right) $ is the set of all distributions $f\in S^{\prime }(\mathbb{H}^{n})$ such that it can be written as 
\begin{equation}
f=\sum\limits_{j=1}^{\infty }\lambda_{j}a_{j}  \label{desc. atomica}
\end{equation}%
in $S^{\prime }(\mathbb{H}^{n}),$ where $\left\{ \lambda_{j}\right\}_{j=1}^{\infty }$ is a sequence of non negative numbers, the $a_{j}$'s are $(p(\cdot),p_{0},D)$ - atoms supported on the $\rho$-ball $B_j$ and 
$\mathcal{A}\left( \left\{ \lambda_{j}\right\}_{j=1}^{\infty },\left\{ B_{j}\right\} _{j=1}^{\infty }, p(\cdot)\right) < \infty$. One defines
\[
\left\Vert f\right\Vert _{H_{atom}^{p(\cdot),p_{0},D}}=\inf \mathcal{A}\left(\left\{ \lambda_{j}\right\} _{j=1}^{\infty }, 
\left\{ B_{j}\right\} _{j=1}^{\infty}, p(\cdot)\right)
\]%
where the infimum is taken over all admissible expressions as in (\ref{desc. atomica}).
\end{definition}

\begin{definition}
Given a collection of sets $\{ E_j \}_{j \in \mathbb{N}}$, we say that the family $\{ E_j \}_{j \in \mathbb{N}}$ has the bounded intersection property if there exists $L \in \mathbb{N}$ such that no point of $\cup_{k \in \mathbb{N}} E_k$ lies in more than $L$ of the sets $E_j$.
\end{definition}

Next, we establish the following version of Theorem 4.5 given in \cite{Fang}. 

\begin{theorem} \label{var decatomic}
If $p(\cdot) \in \mathcal{P}^{\log}(\mathbb{H}^n)$ and $p_0 > \max\{ p_{+}, 1\}$ is sufficiently large, then the quantities $\left\Vert f\right\Vert_{H_{atom}^{p(\cdot), p_{0}, D}}$ and 
$\left\Vert f\right\Vert _{H^{p(\cdot)}}$ are comparable. Moreover, $f$ admits an atomic decomposition 
$f=\sum\limits_{j=1}^{\infty } \lambda_{j} a_{j}$ such that
\[
\mathcal{A}\left( \left\{ \lambda_{j} \right\}_{j=1}^{\infty },\left\{ B_{j}\right\}_{j=1}^{\infty}, p(\cdot)\right) \leq 
C \, \|f \|_{H^{p(\cdot)}},
\]
where $C$ does not depend on $f$ and the family of $\rho$-balls $\left\{ B_{j}\right\} _{j=1}^{\infty}$ has the bounded intersection property.
\end{theorem}

The following two results were proved in \cite[Lemma 5.7 and Proposition 3.3]{rocha2} respectively.

\begin{lemma} \label{ineq p star}
Let $p(\cdot) : \mathbb{H}^{n} \to (0, \infty)$ be an exponent function with $0 < p_{-} \leq p_{+} < \infty$ and let $\{ B_j \}$ be a family of $\rho$ - balls which satisfies the bounded intersection property. If $0 < p_{*} < \underline{p}$, then
\[
\left\| \left\{ \sum_j \left( \frac{\lambda_j \chi_{B_j}}{\| \chi_{B_j} \|_{L^{p(\cdot)}}} \right)^{p_{*}} \right\}^{1/p_{*}}
\right\|_{L^{p(\cdot)}} \sim \mathcal{A} \left( \{ \lambda_j \}_{j=1}^{\infty}, \{ B_j \}_{j=1}^{\infty}, p(\cdot) \right)
\]
for any sequence of nonnegative numbers $\{ \lambda_j \}_{j=1}^{\infty}$.
\end{lemma}

\begin{proposition} \label{b_j functions}
Let $p(\cdot) : \mathbb{H}^{n} \to (0, \infty)$ such that $p(\cdot) \in \mathcal{P}^{\log}(\mathbb{H}^{n})$ and 
$0 < p_{-} \leq p_{+} < \infty$. Let $s > 1$ and $0 < p_{*} < \underline{p}$ such that $s p_{*} > p_{+}$ and let $\{ b_j \}_{j=1}^{\infty}$ be a sequence of nonnegative functions in $L^{s}(\mathbb{H}^{n})$ such that each $b_j$ is supported in a $\rho$ - ball 
$B_j \subset \mathbb{H}^{n}$ and
\begin{equation} \label{bks}
\| b_j \|_{L^{s}(\mathbb{H}^{n})} \leq A_j |B_j|^{1/s},
\end{equation}
where $A_j >0$ for all $j \geq 1$. Then, for any sequence of nonnegative numbers $\{ \lambda_j \}_{j=1}^{\infty}$ we have
\[
\left\| \sum_{j=1}^{\infty} \lambda_j b_j \right\|_{L^{p(\cdot)/p_{*}}(\mathbb{H}^{n})} \leq C \left\| \sum_{j=1}^{\infty} A_j \lambda_j 
\chi_{B_j} \right\|_{L^{p(\cdot)/p_{*}}(\mathbb{H}^{n})},
\]
where $C$ is a positive constant which does not depend on $\{ b_j \}_{j=1}^{\infty}$, $\{ A_j \}_{j=1}^{\infty}$, and 
$\{ \lambda_j \}_{j=1}^{\infty}$.
\end{proposition}
For our next result, we first introduce two maximal operators on $\mathbb{H}^n$. The first is a discrete maximal and the second one 
is a non-tangential maximal. Given $\phi \in \mathcal{S}(\mathbb{H}^n)$, we define
\[
(M_{\phi}^{dis} f)(z) := \sup \left\{ |(f \ast \phi_{2^j})(z)| : j \in \mathbb{Z} \right\},
\]
where $\phi_{2^j}(z) = 2^{-jQ} \phi (2^{-j}z)$, and put
\[
(M_{\phi} f)(z) := \sup \left\{ |(f \ast \phi_t)(w)| : \rho(w^{-1} \cdot z) < t, \, 0 < t < \infty \right\}.  
\]
The following pointwise inequality is obvious
\[
(M_{\phi}^{dis} f)(z) \leq (M_{\phi} f)(z),
\]
for all $z \in \mathbb{H}^n$. Now, this inequality and Proposition 16 in \cite{rocha5} lead to the next result. 

\begin{proposition} \label{Lg dist} Let $g \in L^q_{loc} \cap \mathcal{S}'(\mathbb{H}^n)$ and $f = \mathcal{L} g$ in 
$\mathcal{S}'(\mathbb{H}^n)$. If $\phi \in \mathcal{S}(\mathbb{H}^n)$ and $N > Q+2$, then
\[
(M_{\phi}^{dis}f)(z) \leq C \| \phi \|_{\mathcal{S}(\mathbb{H}^{n}), N}  \,\,\, \eta_{q, 2}(g; \, z)
\]
holds for all $z \in \mathbb{H}^n$.
\end{proposition}

A fundamental solution for the sublaplacian on $\mathbb{H}^n$ was obtained by G. Folland in \cite{Folland}. More precisely, he proved the following result.

\begin{theorem} \label{fund solution}
$c_n \, \rho^{-2n}$ is a fundamental solution for $\mathcal{L}$ with source at $0$, where
\[
\rho(x,t) = (|x|^4 + t^2)^{1/4},
\]
and
\[
c_n = \left[ n(n+2) \int_{\mathbb{H}^n} |x|^2 (\rho(x,t)^4 + 1)^{-(n+4)/2} dxdt \right]^{-1}.
\]
In others words, for any $u \in \mathcal{S}(\mathbb{H}^{n})$, $\left( \mathcal{L}u,  c_n \rho^{-2n} \right) = u(0)$.
\end{theorem}

If $a$ is a bounded function with compact support on $\mathbb{H}^n$, its potential $b$, defined as 
\[
b(z) := \left( a \ast c_n \, \rho^{-2n} \right)(z) = c_n \int_{\mathbb{H}^{n}} \rho(w^{-1} \cdot z)^{-2n} a(w) dw,
\]
is a locally bounded function and, by Theorem \ref{fund solution}, $\mathcal{L} b = a$ in the sense of distributions. For these potentials, we have the following result.

\

In the sequel, $\beta$ is the constant in \cite[Corollary 1.44]{Foll-St}, we observe that $\beta \geq 1$ (see \cite[p. 29]{Foll-St}).

\begin{lemma} \label{b function} Let $a(\cdot)$ be an $(p(\cdot), p_{0}, D)-$atom centered at the $\rho-$ball $B=B(z_0, \delta)$. If $b(z) = (a \ast c_n \rho^{-2n})(z)$, then for $\rho(z_0^{-1} z) \geq 2 \beta^{2}\delta$ and every multi-index $I$ there exists a positive constant $C_{I}$ such that 
\[
\left| (X^{I}b)(z) \right| \leq C_{I} \, \delta^{2+Q} \, \|\chi_B \|^{-1}_{p(\cdot)} \, \rho(z_{0}^{-1} \cdot z)^{-Q-d(I)}
\]
holds.
\end{lemma}

\begin{proof} 
The proof is similar to the one given in \cite[Lemma 19]{rocha5}, but considering now the Definition \ref{atom} above.
\end{proof}

\begin{proposition} \label{main estimate}
Let $a(\cdot)$ be an $(p(\cdot), p_{0}, D)-$atom centered at the $\rho-$ball 
$B=B(z_0, \delta)$. If $b(z) = (a \ast c_n \rho^{-2n})(z)$, then for all $z \in \mathbb{H}^{n}$
\begin{eqnarray} 
\label{Nq2b}
N_{q, 2} \left(\widetilde{b}; \, z \right) &\lesssim& \| \chi_{B} \|_{p(\cdot)}^{-1} \left[(M \chi_{B})(z) \right]^{\frac{2 + Q/q}{Q}} + 
\chi_{4 \beta^2 B}(z) (M a)(z)  \\
\notag
&+& \chi_{4 \beta^2 B}(z) \sum_{d(I)=2} (T^{*}_{I} a)(z),
\end{eqnarray}
where $\widetilde{b}$ is the class of $b$ in $E^{q}_{1}$, $M$ is the Hardy-Littlewood maximal operator and $(T^{*}_{I} a) (z) = \sup_{\epsilon >0} \left|\int_{\rho(w^{-1} \cdot z) > \epsilon} \, (X^{I} \rho^{-2n})(w^{-1} \cdot z) a(w) \, dw \right|$.
\end{proposition}

\begin{proof} We point out that the argument used in the proof of Proposition 20 in \cite{rocha5}, to obtain the pointwise inequality
(16) there, works in this setting as well, but considering now the conditions $a1)$, $a2)$ and $a3)$ given in Definition 
\ref{atom} of $(p(\cdot), p_0, D)-$atom. These conditions are similar to those of the atoms in classical context (see p. 71-72 
in \cite{Foll-St}). Then, this observation and Lemma \ref{b function} allow us to get (\ref{Nq2b}).
\end{proof}

\section{Main results}

In this section we will prove our main results.

\begin{theorem} \label{main thm} 
Let $p(\cdot)$ be an exponent that belongs to $\mathcal{P}^{\log}(\mathbb{H}^{n})$ and $1 < q < \frac{n+1}{n}$. 
If $\underline{p} > Q (2 + Q/q)^{-1}$, then the sublaplacian $\mathcal{L}$ is a bijective mapping from 
$\mathcal{H}^{p(\cdot)}_{q, 2}(\mathbb{H}^{n})$ onto $H^{p(\cdot)}(\mathbb{H}^{n})$. Moreover, there exist two positive constants 
$c_1$ and $c_2$ such that
\[
c_1 \|G \|_{\mathcal{H}^{p(\cdot)}_{q, 2}(\mathbb{H}^{n})} \leq \| \mathcal{L}G \|_{H^{p(\cdot)}(\mathbb{H}^{n})} \leq 
c_2 \|G \|_{\mathcal{H}^{p(\cdot)}_{q, 2}(\mathbb{H}^{n})}
\]
hold for all $G \in \mathcal{H}^{p(\cdot)}_{q, 2}(\mathbb{H}^{n})$.
\end{theorem}

\begin{proof} Let $G \in \mathcal{H}^{p(\cdot)}_{q, \, 2}(\mathbb{H}^{n})$. Since $N_{q, 2}(G; z)$ is finite 
$\text{a.e.} \,\, z \in \mathbb{H}^{n}$,  by $(i)$ in Lemma \ref{puntual 1} and Proposition \ref{g distrib}, the unique representative 
$g$ of $G$ (which depends on $z$), satisfying $\eta_{q, 2}(g; z)=N_{q, 2}(G; z)$, is a function in 
$L^{q}_{loc}(\mathbb{H}^{n}) \cap \mathcal{S}'(\mathbb{H}^{n})$. In particular, for a radial function $\phi \in \mathcal{S}(\mathbb{H}^n)$ with $\int \phi = 1$, by Proposition \ref{Lg dist}, we get
\[
M_{\phi}^{dis}(\mathcal{L}G)(z) \leq C_{\phi} \,\, N_{q, \, 2}(G; z).
\]
Thus, this inequality and Theorem 3.2 in \cite{Fang} give $\mathcal{L}G \in H^{p(\cdot)}(\mathbb{H}^{n})$ and 
\begin{equation} \label{continuity}
\| \mathcal{L}G \|_{H^{p(\cdot)}(\mathbb{H}^{n})} \leq C \, \| G \|_{\mathcal{H}^{p(\cdot)}_{q, \, 2}(\mathbb{H}^{n})}.
\end{equation}
This proves the continuity of sublaplacian $\mathcal{L}$ from $\mathcal{H}^{p(\cdot)}_{q, \, 2}(\mathbb{H}^{n})$ into 
$H^{p(\cdot)}(\mathbb{H}^{n})$.

Now we shall see that the operator $\mathcal{L}$ is onto. By Theorem \ref{var decatomic}, given $f \in H^{p(\cdot)}(\mathbb{H}^{n})$ there exists a sequence of nonnegative numbers $\{ \lambda_j \}_{j=1}^{\infty}$ and a sequence of $\rho-$balls $\{B_j \}_{j=1}^{\infty}$ (with the bounded intersection property) and $(p(\cdot), p_0, D)-$atoms $a_j$ centered at $B_j$, such that $f= \sum_{j=1}^{\infty} \lambda_j a_j$ and
\[
\mathcal{A}\left( \left\{ \lambda_{j}\right\}_{j=1}^{\infty },\left\{ B_{j}\right\} _{j=1}^{\infty },  p(\cdot)\right) 
\lesssim \|f \|_{H^{p(\cdot)}(\mathbb{H}^{n})}.
\]
For each $j \in \mathbb{N}$ we put
$b_j(z)= (a_j \ast c_n \rho^{-2n})(z)$, from Proposition \ref{main estimate} we have
\begin{eqnarray*} 
\label{N estimate}
N_{q, 2} \left(\widetilde{b}_j; \, z \right) &\lesssim& \| \chi_{B_j} \|_{p(\cdot)}^{-1} \left[(M \chi_{B_j})(z) \right]^{\frac{2 + Q/q}{Q}} + \chi_{4 \beta^2 B}(z) (M a_j)(z)  \\
&+& \chi_{4 \beta^2 B_j}(z) \sum_{d(I)=2} (T^{*}_{I} a_j)(z)
\end{eqnarray*}
Thus
\begin{eqnarray*}
\sum_{j=1}^{\infty} k_j N_{q, 2} \left(\widetilde{b}_j; \, z \right) &\lesssim& \sum_{j=1}^{\infty} \lambda_j 
\| \chi_{B_j} \|_{p(\cdot)}^{-1} 
\left[(M \chi_{B_j})(z) \right]^{\frac{2 + Q/q}{Q}} \\
&+& \sum_{j=1}^{\infty} \lambda_j \chi_{4 \beta^2 B_j}(z) (M a_j)(z) \\
&+& \sum_{j=1}^{\infty} \lambda_j \chi_{4 \beta^2 B_j}(z) \sum_{d(I)=2} (T^{*}_{I} a_j)(z) \\
&=& I + II + III.
\end{eqnarray*}
To study $I$, by hypothesis, we have that $(2 + Q/q) \underline{p} > Q $. Then
\[
\|I\|_{p(\cdot)} = \left\| \sum_{j=1}^{\infty} \lambda_j \| \chi_{B_j} \|_{p(\cdot)}^{-1} 
\left[(M\chi_{B_j})(\cdot) \right]^{\frac{2 + Q/q}{Q}} \right\|_{p(\cdot)}
\]
\[
= \left\| \left\{ \sum_{j=1}^{\infty} \lambda_j \| \chi_{B_j} \|_{p(\cdot)}^{-1} \left[(M\chi_{B_j})(\cdot) \right]^{\frac{2 + Q/q}{Q}} \right\}^{\frac{Q}{2 + Q/q}} \right\|_{\frac{2 + Q/q}{Q}p(\cdot)}^{\frac{2 + Q/q}{Q}}
\]
\[
\lesssim \left\| \left\{ \sum_{j=1}^{\infty} \lambda_j \| \chi_{B_j} \|_{p(\cdot)}^{-1} \chi_{B_j}(\cdot)  \right\}^{\frac{Q}{2 + Q/q}} \right\|_{\frac{2 + Q/q}{Q}p(\cdot)}^{\frac{2 + Q/q}{Q}}
\]
\[
= \left\| \sum_{j=1}^{\infty} \lambda_j \| \chi_{B_j} \|_{p(\cdot)}^{-1} 
\chi_{B_j}(\cdot) \right\|_{p(\cdot)}
\]
\[
\lesssim \left\| \left\{ \sum_{j=1}^{\infty} \left( \lambda_j \| \chi_{B_j} \|_{p(\cdot)}^{-1} \chi_{B_j}(\cdot) \right)^{\underline{p}}  \right\}^{1/ \underline{p}} \right\|_{p(\cdot)} 
\]
\[
= \mathcal{A}\left( \left\{ \lambda_{j}\right\}_{j=1}^{\infty },\left\{ B_{j}\right\} _{j=1}^{\infty }, p(\cdot)\right)
\lesssim \|f \|_{H^{p(\cdot)}},
\]
where the first inequality follows from Theorem \ref{vector max ineq}, since $(2 + Q/q) \underline{p} > Q $ and $(2+ Q/q)/Q > 1$. The embedding 
$\ell^{\underline{p}} = \ell^{\min \{ p_{-}, 1 \}} \hookrightarrow \ell^{1}$ gives the second inequality.

\qquad

To study $II$, let $0 < p_{*} < \underline{p}$ be fixed and $p_{0} > \max \{ p_{+}, 1 \}$. Since the Hardy-Littlewood maximal operator $M$ is bounded on $L^{p_0}(\mathbb{H}^n)$ (see \cite[Theorem 1, p. 13]{Elias}), we have
\[
\|  (Ma_j)^{p_{*}} \|_{L^{p_{0}/p_{*}}(4 \beta B_j)} \lesssim \| a_j \|_{p_0}^{p_{*}} \lesssim 
\frac{ |B_j |^{p_{*}/p_0}}{\| \chi_{B_j} \|_{p(\cdot)}^{p_{*}}} 
\lesssim \frac{ |4\beta B_j |^{p_{*}/p_0}}{\| \chi_{4\beta B_j} \|_{p(\cdot)/p_{*}}},
\]
where the third inequality holds since the quantities $\| \chi_{4\beta B_j} \|_{p(\cdot)}$ and $\| \chi_{B_j} \|_{p(\cdot)}$ are 
comparable (see Lemma \ref{Lp size}). Now, since $0 < p_{*} < 1$, we apply the $p_{\ast}$-inequality and Proposition 
\ref{b_j functions} with $b_j = \left( \chi_{4\beta B_{j}} \, (Ma_j)^{p_{*}} \right)$, $A_j = \left\| \chi_{4\beta B_{j}} \right\|_{L^{p(\cdot)/p_{*}}}^{-1}$ and $s= p_0/p_{*}$, to obtain
\begin{eqnarray*}
\| II \|_{L^{p(\cdot)}} &\lesssim& \left\| \sum_{j} \left(\lambda_j \, \chi_{4\beta B_{j}} \, Ma_j \right)^{p_{*}} 
\right\|^{1/p_{*}}_{L^{p(\cdot)/p_{*}}} \\
&\lesssim& \left\| \sum_{j} \left( \frac{\lambda_j}{\left\| \chi_{4\beta B_{j}} \right\|_{L^{p(\cdot)}}} \right)^{p_{*}} 
\chi_{4\beta B_{j}} \right\|^{1/p_{*}}_{L^{p(\cdot)/p_{*}}}.
\end{eqnarray*}
It is easy to check that $\chi_{4\beta B_{j}} \leq (M\chi_{B_j})^{2}$. From this inequality, Theorem \ref{vector max ineq} and Lemma \ref{Lp size}, we have
\begin{eqnarray*}
\| II \|_{L^{p(\cdot)}} &\lesssim& \left\| \left\{ \sum_{j} \left( \frac{\lambda_j^{p_{*}/2}}{\left\| \chi_{B_{j}} 
\right\|^{p_{*}/2}_{L^{p(\cdot)}}} (M\chi_{B_j}) \right)^{2}  \right\}^{1/2} \right\|^{2/p_{*}}_{L^{2p(\cdot)/p_{*}}} \\
&\lesssim& \left\| \left\{ \sum_{j} \left( \frac{ \lambda_j \chi_{B_j}}{\| \chi_{B_j} \|_{L^{p(\cdot)}}} \right)^{p_{*}} 
\right\}^{1/p_{*}} \right\|_{L^{p(\cdot)}}.
\end{eqnarray*}
Finally, Lemma \ref{ineq p star} gives
\[
\| II \|_{L^{p(\cdot)}} \lesssim \mathcal{A}\left( \{ \lambda_j \}_{j=1}^{\infty}, \{ B_j \}_{j=1}^{\infty}, p(\cdot) \right) 
\lesssim \| f \|_{H^{p(\cdot)}}.
\]

Now, we study $III$, by Theorem 3 in \cite{Folland} and Corollary 2, p. 36, in \cite{Elias} (see also {\bf 2.5}, p. 11, in \cite{Elias}), 
we have, for every multi-index $I$ with $d(I) = 2$, that the operator $T_{I}^{*}$ is bounded on $L^{p_0}(\mathbb{H}^n)$ for every 
$1 < p_0 < \infty$. Proceeding as in the estimate of $II$, we get
\[
\| III\|_{L^{p(\cdot)}} \lesssim \mathcal{A}\left( \{ \lambda_j \}_{j=1}^{\infty}, \{ B_j \}_{j=1}^{\infty}, p(\cdot) \right) 
\lesssim \| f \|_{H^{p(\cdot)}}.
\]
Thus, we have
$$\left\| \sum_{j=1}^{\infty} \lambda_j N_{q, 2}(\widetilde{b}_j; \, \cdot) \right\|_{p(\cdot)} \lesssim \|f \|_{H^{p(\cdot)}}.$$
By Lemma \ref{modular}, we obtain $\kappa_{p(\cdot)}\left( \sum_{j=1}^{\infty} \lambda_j N_{q, 2}(\widetilde{b}_j; \cdot) \right) < \infty$. Hence
\begin{equation}
\sum_{j=1}^{\infty} \lambda_j N_{q, 2}(\widetilde{b}_j; \, z) < \infty \,\,\,\,\,\, \text{a.e.} \, z \in \mathbb{H}^{n} \label{Nq}
\end{equation}
and
\begin{equation}
\kappa_{p(\cdot)}\left( \sum_{j=M+1}^{\infty} \lambda_j N_{q, 2}(\widetilde{b}_j; \, \cdot) \right) \rightarrow 0, \,\,\,\, \text{as} \,\,  M \rightarrow \infty  \label{Nq2}.
\end{equation}
From (\ref{Nq}) and Lemma \ref{series in Eqk}, there exists a function $G$ such that $\sum_{j=1}^{\infty} \lambda_j \widetilde{b}_j = G$ 
in $E^{q}_{1}$ and
\[
N_{q, 2} \left( \left(G - \sum_{j=1}^{M} \lambda_j \widetilde{b}_j \right) ; \, z \right) \leq c \, 
\sum_{j=M+1}^{\infty} \lambda_j N_{q, 2}(\widetilde{b}_j; z).
\]
This estimate together with (\ref{Nq2}) and Lemma \ref{modular a cero} implies
\[
\left\| G - \sum_{j=1}^{M} \lambda_j \widetilde{b}_j \right\|_{\mathcal{H}^{p(\cdot)}_{q,2}} \rightarrow 0, \,\,\,\, \text{as} \,\,  M \rightarrow \infty.
\]
By proposition \ref{completeness}, we have that $G \in \mathcal{H}^{p(\cdot)}_{q,2}(\mathbb{H}^{n})$ and 
$G = \sum_{j=1}^{\infty} \lambda_j \widetilde{b}_j$ in $\mathcal{H}^{p(\cdot)}_{q,2}(\mathbb{H}^{n})$. Since $\mathcal{L}$ is a continuous operator from $\mathcal{H}^{p(\cdot)}_{q,2}(\mathbb{H}^{n})$ into $H^{p(\cdot)}(\mathbb{H}^{n})$, we get
\[
\mathcal{L}G = \sum_j \lambda_j \mathcal{L} \widetilde{b}_j = \sum_j \lambda_j a_j = f,
\]
in $H^{p(\cdot)}(\mathbb{H}^{n})$. This shows that $\mathcal{L}$ is onto $H^{p(\cdot)}(\mathbb{H}^{n})$. Moreover,
\[
\| G\|_{\mathcal{H}^{p(\cdot)}_{q,2}} = \left\| \sum_{j=1}^{\infty} \lambda_j \widetilde{b}_j \right\|_{\mathcal{H}^{p(\cdot)}_{q,2}} \lesssim \left\| \sum_{j=1}^{\infty} \lambda_j N_{q, 2}(\widetilde{b}_j; \cdot) \right\|_{p(\cdot)} \lesssim \|f \|_{H^{p(\cdot)}} = \| \mathcal{L} G \|_{H^{p(\cdot)}}.
\]
For to conclude the proof, we will show that the operator $\mathcal{L}$ is injective on $\mathcal{H}^{p(\cdot)}_{q, \, 2}$. Let 
$\mathcal{O} = \{ z : N_{q, 2m}(G; z) > 1\}$. The set $\mathcal{O}$ is open because of that $N_{q, 2}(G; \cdot)$ is lower semicontinuous.
We take a constant $r>0$  such that $r \geq \max\{q, p_{+}\}$. Since $N_{q, 2}(G; \cdot) \in L^{p(\cdot)}(\mathbb{H}^{n})$ it follows that 
$|\mathcal{O}|$ is finite and $N_{q, 2}(G; \cdot) \in L^{r}(\mathbb{H}^{n} \setminus \mathcal{O})$ thus following the proof of Theorem 18 in \cite{rocha5} we have that if $G \in \mathcal{H}^{p(\cdot)}_{q, 2}(\mathbb{H}^{n})$ and $\mathcal{L} G = 0$, then $G \equiv 0$. This proves the injectivity of $\mathcal{L}$. 
\end{proof}

Therefore, Theorem \ref{main thm} allows us to conclude, for $Q(2+Q/q)^{-1} < \underline{p}$, that the equation
\[ 
\mathcal{L} F = f, \,\,\,\,\,\, f \in H^{p(\cdot)}(\mathbb{H}^n) 
\]
has a unique solution in $\mathcal{H}^{p(\cdot)}_{q, 2}(\mathbb{H}^{n})$, namely: $F := \mathcal{L}^{-1}f$.

\

We shall now see that the case $0 < p_{+} \leq Q \, (2 + \frac{Q}{q})^{-1}$ is trivial.

\begin{theorem} \label{second thm}
If $p(\cdot)$ is an exponent function on $\mathbb{H}^n$ such that $p_{+} \leq Q (2 + Q/q)^{-1}$, then 
$\mathcal{H}^{p(\cdot)}_{q, \, 2}(\mathbb{H}^{n}) = \{ 0 \}.$
\end{theorem}

\begin{proof} Let $G \in \mathcal{H}^{p(\cdot)}_{q, \, 2}(\mathbb{H}^{n})$ and assume $G \neq 0$. Then there exists 
$g \in G$ that is not a polynomial of homogeneous degree less or equal to $1$. It is easy to check that there exist a positive constant $c$ and a $\rho-$ball $B = B(e,r)$ with $r>1$ such that
\[
\int_{B} |g(w) - P(w)|^{q} \, dw \geq c > 0,
\]
for every $P \in \mathcal{P}_{1}$.

Let $z$ be a point such that $\rho(z) > r$ and let $\delta = 2 \rho(z)$. Then $B(e, r) \subset B(z, \delta)$. If $h \in G$, then 
$h = g - P$ for some $P \in \mathcal{P}_{1}$ and
\[
\delta^{-2}|h|_{q, B(z, \delta)} \geq c \rho(z)^{-2-Q/q}.
\]
So $N_{q,2}(G; \, z) \geq c \, \rho(z)^{-2-Q/q}$, for $\rho(z) > r$. Since $p_{+} \leq Q(2+Q/q)^{-1}$, we have that
\[
\kappa_{p(\cdot)} \left( N_{q,2m}(G; \cdot) \right) \geq c \, \int_{\rho(z) > r} \rho(z)^{-(2+Q/q)p_{+}} \, dz = \infty,
\]
which gives a contradiction. Thus $\mathcal{H}^{p}_{q, 2}(\mathbb{H}^{n}) = \{0\}$, if $p_{+} \leq Q(2+Q/q)^{-1}$.
\end{proof}

\

Pablo Rocha, Instituto de Matem\'atica (INMABB), Departamento de Matem\'atica, Universidad Nacional del Sur (UNS)-CONICET, Bah\'ia Blanca, Argentina. \\
{\it e-mail:} pablo.rocha@uns.edu.ar

\end{document}